\documentclass[abbrev]{amsart}

\usepackage{amsmath,amssymb}
\usepackage{amsthm}
\usepackage{amsrefs}
\usepackage{indentfirst}
\usepackage{mathrsfs}
\usepackage{hyperref}
\usepackage{xcolor}
\usepackage[margin=1.4in]{geometry}
\usepackage{nicefrac}
\usepackage{graphics}
\usepackage{tikz}
\usepackage{bm}

\newcommand{\bR}{\mathbb{R}}
\newcommand{\bfx}{\mathbf{x}}
\newcommand{\bfw}{\mathbf{w}}
\newcommand{\bfy}{\mathbf{y}}
\newcommand{\bmeta}{\bm{\eta}}
\newcommand{\bmphi}{\bm{\phi}}
\newcommand{\bmpsi}{\bm{\psi}}
\newcommand{\bfe}{\mathbf{e}}

\newcommand{\bmxi}{\bm{\xi}}

\newlength{\leftstackrelawd}
\newlength{\leftstackrelbwd}
\def\leftstackrel#1#2{\settowidth{\leftstackrelawd}%
{${{}^{#1}}$}\settowidth{\leftstackrelbwd}{$#2$}%
\addtolength{\leftstackrelawd}{-\leftstackrelbwd}%
\leavevmode\ifthenelse{\lengthtest{\leftstackrelawd>0pt}}%
{\kern-.5\leftstackrelawd}{}\mathrel{\mathop{#2}\limits^{#1}}}

\numberwithin{equation}{section}
\newtheorem{theorem}{Theorem}
\newtheorem{lemma}{Lemma}[section]
\newtheorem{remark}[lemma]{Remark}
\newtheorem{proposition}[lemma]{Proposition}
\newtheorem{assumption}{Assumption}

\allowdisplaybreaks[4]

\title[Representation of Anti-Symmetric Functions]{Exact and Efficient Representation of Totally Anti-Symmetric Functions}
\author{Ziang Chen}
\address{(ZC) Department of Mathematics, Massachusetts Institute of Technology, 77 Massachusetts Avenue, Cambridge, MA 02139.}
\email{ziang@mit.edu}
\author{Jianfeng Lu}
\address{(JL) Departments of Mathematics, Physics, and Chemistry, Duke University, Box 90320, Durham, NC 27708.}
\email{jianfeng@math.duke.edu}
\date{\today}

\thanks{The work of ZC is supported in part by the National Science Foundation via award DMS-2509011. The work of JL is supported in part by National Science Foundation via awards DMS-2012286 and DMS-2309378. We thank Jiequn Han for some helpful discussions.}

\thanks{Corresponding author: Ziang Chen (ziang@mit.edu)}

\begin{document}

\begin{abstract}
This paper concerns the long-standing question of representing (totally) anti-symmetric functions in high dimensions. 
We propose a new ansatz based on the composition of an odd function with a fixed set of anti-symmetric basis functions. We prove that this ansatz can exactly represent every anti-symmetric and continuous function and the number of basis functions has efficient scaling with respect to dimension (number of particles). The singular locus of the anti-symmetric basis functions is precisely identified.
\end{abstract}

\maketitle

\section{Introduction}

Let $d\geq 1$, $n\geq 1$, and $\Omega\subset\bR^d$. We say that a function $f:\Omega^n\to\bR$ is (totally) symmetric if
\begin{equation}\label{eq:def_sym}
    f(\bfx_{\sigma(1)},\bfx_{\sigma(2)},\dots,\bfx_{\sigma(n)}) = f(\bfx_1,\bfx_2,\dots,\bfx_n),\quad\forall~\sigma\in S_n,
\end{equation}
where $S_n$ is the group of all permutations on $\{1,2,\dots,n\}$, and is (totally) anti-symmetric if
\begin{equation}\label{eq:def_antisym}
    f(\bfx_{\sigma(1)},\bfx_{\sigma(2)},\dots,\bfx_{\sigma(n)}) = (-1)^\sigma f(\bfx_1,\bfx_2,\dots,\bfx_n),\quad\forall~\sigma\in S_n,
\end{equation}
where $(-1)^\sigma$ is the signature of the permutation $\sigma$. 

This work concerns the representation of anti-symmetric functions, which is a problem that dates back at least to Cauchy \cite{Cauchy:1815} and has received intense research over the years due to its ultimate importance in the study of quantum many-body systems. Indeed, according to the Pauli exclusion principle, the many-body wavefunctions for any identical fermions have to be anti-symmetric, and thus it has been a long-standing problem for finding efficient ansatz for anti-symmetric functions, in the literature of chemistry, physics, and related fields. 

Conventional approaches for the representation of anti-symmetric functions center around Slater determinants, which is essentially the starting point of quantum chemistry, known as the Hartree-Fock method. Many attempts have been made to generalize the ansatz beyond Hartree-Fock, including Slater-Jastrow wave functions, the backflow transformation \cite{feynman1956energy}, etc. In particular, inspired by the recent advances in machine learning for approximating high dimensional function using neural networks, there has been a surge of activities in recent years exploring the ansatz based on neural networks, see e.g., \cites{acevedo2020vandermonde,carleo2017solving,choo2020fermionic,han2019solving,hermann2020deep,hutter2020representing,luo2019backflow,pfau2020ab,schatzle2021convergence,spencer2020better,barrett2022autoregressive,li2022ab,von2022self,pescia2022neural, robledo2022fermionic}. 

By far most of the attempts are based on linear combinations of determinants, such as the back-flow transformation. The representational power of such ansatz has been analyzed in some recent works \cites{han2019universal, hutter2020representing, zweig2022towards}. 
Unfortunately, a recent negative result obtained by one of us with collaborators \cite{huang2023geometry} shows that an exponential number (wrt number of particles in the system) of determinants is needed in general. This calls for new ansatzes for anti-symmetric functions going beyond the linear combination of determinants, still popular in the application literature.

In this work, based on this line of thought, we propose a new ansatz for anti-symmetric functions, inspired by our recent result for symmetric case \cite{chen2022representation}, where we proved the universal representation for symmetric function based on the composition of a multi-variable function with generators of symmetric polynomials \cites{briand2004algebra, weyl1946classical}, and a more recent work \cite{wang2023polynomial} that use symmetric functions distinguishing different orbits rather than generators of symmetric polynomials. % Our ansatz is based on the composition of an odd function with a set of anti-symmetric basis functions, as will be made more precise in the next section. 
As the main result of this work, %we establish that our ansatz gives an exact representation of anti-symmetric and continuous functions. 
we establish that there exists a set of anti-symmetric basis functions, such that any anti-symmetric and continuous function can be expressed uniquely as the composition of a continuous odd function with the set of basis functions, which will be stated more precisely in the next section. Furthermore, our ansatz is efficient in the sense that the set of basis functions does not suffer from the curse of dimensionality (CoD), i.e., its cardinality scales at most polynomially in the dimension, which is of essential importance in terms of computational complexity. To the best of our knowledge, this is the first such representation for anti-symmetric functions in high dimensions. We also establish that the singular locus of some explicit basis functions is precisely the set of configurations with at least two identical particles, and show with examples that the regularity of the odd function in the representation may not be as good as the original anti-symmetric function.

\section{Results}

\subsection{Representation theorem for anti-symmetric functions}
We first state the required properties of our anti-symmetric basis functions.

\begin{assumption}\label{asp:eta}
    Given $d,n\geq 1$ and $\Omega\subseteq\bR^d$. Assume that $\bmeta = (\eta_1,\eta_2,\dots,\eta_m):\Omega^n\to\bR^m$ satisfies the followings:
    \begin{itemize}
        \item[(i)] $\eta_k:\Omega^n\to\bR$ is anti-symmetric and continuous for each $k\in\{1,2,\dots,m\}$.
        \item[(ii)] $\bmeta(\bfx_1,\bfx_2,\dots,\bfx_n) = 0$ if and only if $\bfx_i=\bfx_j$ for some $i,j\in\{1,2,\dots,n\}$ with $i\neq j$.
        \item[(iii)] If $\bmeta(\bfx_1,\bfx_2,\dots,\bfx_n) = \bmeta(\bfx_1',\bfx_2',\dots,\bfx_n') \neq 0$, then there exists a permutation $\sigma\in S_n$ such that $(\bfx_1',\bfx_2',\dots,\bfx_n') = (\bfx_{\sigma(1)},\bfx_{\sigma(2)},\dots,\bfx_{\sigma(n)})$.
    \end{itemize}
\end{assumption}

\begin{remark}\label{rmk:x_i=x_j}
    For any anti-symmetric function $f:\Omega^n\to\bR$, one always has $f(\bfx_1,\bfx_2,\dots,\bfx_n) = 0$ if $\bfx_i=\bfx_j$ for some $i\neq j$, which follows directly from \eqref{eq:def_antisym} by considering $\sigma\in S_n$ that switches $i$ and $j$.
\end{remark}

In Assumption~\ref{asp:eta}, Condition (ii) states that $\bmeta$ can identify $(\bfx_1,\bfx_2,\dots,\bfx_n)$ for which any two vectors differ, and Condition (iii) requires that $\bmeta$ has sufficient separation power to distinguish different orbits of $S_n$ as long as they represent sets of distinct vectors. The construction of $\bmeta$ satisfying Assumption~\ref{asp:eta} will be discussed in Section~\ref{sec:construct_eta} and our main representation theorem for anti-symmetric functions is as follows.

\begin{theorem}\label{thm:main}
    Given $d,n\geq 1$, let $\Omega\subset\bR^d$ be compact and let $\bmeta:\Omega^n\to\bR^m$ satisfy Assumption~\ref{asp:eta}. For any anti-symmetric continuous function $f:\Omega^n\to\bR$, there exists a unique $g:\bmeta(\Omega^n)\to \bR$ that is continuous and odd, and satisfies
    \begin{equation}\label{eq:decomp}
        f(\bfx_1,\bfx_2,\dots,\bfx_n) = g\left(\bmeta(\bfx_1,\bfx_2,\dots,\bfx_n)\right),\quad\forall~\bfx_1,\bfx_2,\dots,\bfx_n\in \Omega,
    \end{equation}
    where $\bmeta(\Omega^n)$ is equipped with the topology induced from $\bR^m$.
\end{theorem}

Theorem~\ref{thm:main} essentially states that with a set $\bmeta$ of anti-symmetric basis functions, the problem of representing anti-symmetric functions can be reduced to representing odd functions. Let us remark that odd functions are much easier to represent since the group $\{\pm 1\}$ is of very small size which allows us to do efficient anti-symmetrization. More explicitly, one can take any continuous $h:\bmeta(\Omega^n)\to\bR$ and set $g(\bfy) = \frac{h(\bfy) - h(-\bfy)}{2}$.

\begin{proof}[Proof of Theorem~\ref{thm:main}]
    For any $\bfy\in \bmeta(\Omega^n)$ and any $(\bfx_1,\bfx_2,\dots,\bfx_n)\in\bmeta^{-1}(\bfy)$, define
    \begin{equation}\label{eq:def_g}
        g(\bfy) = f(\bfx_1,\bfx_2,\dots,\bfx_n),
    \end{equation}
    which is unique according to \eqref{eq:decomp}. We first verify that the function $g:\bR^m\to\bR$ in \eqref{eq:def_g} is well-defined. Suppose that $(\bfx_1,\bfx_2,\dots,\bfx_n),(\bfx_1',\bfx_2',\dots,\bfx_n')\in\bmeta^{-1}(\bfy)$ for some $\bfy\in \bmeta(\Omega^n)$. If $\bfy = 0$, then by Assumption~\ref{asp:eta} (ii), there exists $i\neq j$ such that $\bfx_i=\bfx_j$, which leads to $f(\bfx_1,\bfx_2,\dots,\bfx_n) = 0$. By the same argument, we also have $f(\bfx_1',\bfx_2',\dots,\bfx_n') = 0$. If $\bfy \neq 0$, by Assumption~\ref{asp:eta} (iii), there exists a permutation $\sigma\in S_n$ such that $(\bfx_1',\bfx_2',\dots,\bfx_n') = (\bfx_{\sigma(1)},\bfx_{\sigma(2)},\dots,\bfx_{\sigma(n)})$. Thus, it holds that
    \begin{equation*}
        \bfy = \bmeta(\bfx_1',\bfx_2',\dots,\bfx_n') = (-1)^\sigma \bmeta(\bfx_1,\bfx_2,\dots,\bfx_n) = (-1)^\sigma \bfy,
    \end{equation*}
    which leads to $(-1)^\sigma=1$ since $\bfy\neq 0$. One can then conclude that
    \begin{equation*}
        f(\bfx_1',\bfx_2',\dots,\bfx_n') = (-1)^\sigma f(\bfx_1,\bfx_2,\dots,\bfx_n) = f(\bfx_1,\bfx_2,\dots,\bfx_n).
    \end{equation*}

    Furthermore, the function $g$ defined in \eqref{eq:def_g} is odd since for any $\sigma\in S_n$ with $(-1)^\sigma = -1$, one has
    \begin{align*}
        g( - \bmeta(\bfx_1,\bfx_2,\dots,\bfx_n)) & = g(\bmeta(\bfx_{\sigma(1)},\bfx_{\sigma(2)},\dots,\bfx_{\sigma(n)})) =  f(\bfx_{\sigma(1)},\bfx_{\sigma(2)},\dots,\bfx_{\sigma(n)}) \\
        & = - f(\bfx_1,\bfx_2,\dots,\bfx_n) = - g(\bmeta(\bfx_1,\bfx_2,\dots,\bfx_n)).
    \end{align*}
    
    Finally, we prove that $g$ is continuous. Let $E\subset\bR$ be any closed subset of $\bR$. The continuity of $f$ implies that $f^{-1}(E)$ is a closed subset of $\Omega^n$. Noticing that $\Omega$ is compact, we have that $f^{-1}(E)$ is compact. Therefore, $g^{-1}(E) = \bmeta(f^{-1}(E))$ is also compact by the continuity of $\bmeta$, which implies that $g^{-1}(E)$ is closed.  The function $g$ is thus continuous.
\end{proof}

The proof idea of Theorem~\ref{thm:main} is similar to that of \cite{dym2024low}*{Proposition 2.3} and \cite{chen2022representation}*{Theorem 2.1}, in the sense that all these works define $g$ by taking the same value (up to a sign) on an orbit or union of orbits. However, there are still some crucial differences. In particular, \cites{dym2024low,chen2022representation} employ the quotient space of the group action, which is not available in our setting due to different signs on the same orbit. Additionally, we have to combine all orbits with at least two identical particles due to the special property of antisymmetric functions described in Remark~\ref{rmk:x_i=x_j}, while the basis functions in \cites{dym2024low,chen2022representation} induce an injective map on the quotient.

At the end of this subsection, we comment on the comparison with another representation in \cite{hutter2020representing} that proves that any antisymmetric function $f:\Omega^n\to\bR$ can be represented as
\begin{equation*}
    f(\bfx) = \text{det}\begin{pmatrix}
        g_1(\bfx_1 | \bfx_{\neq 1}) & g_2(\bfx_1 | \bfx_{\neq 1}) & \cdots & g_n(\bfx_1 | \bfx_{\neq 1}) \\
        g_1(\bfx_2 | \bfx_{\neq 2}) & g_2(\bfx_2 | \bfx_{\neq 2}) & \cdots & g_n(\bfx_2 | \bfx_{\neq 2}) \\
        \vdots & \vdots & \ddots & \vdots \\
        g_1(\bfx_n | \bfx_{\neq n}) & g_2(\bfx_n | \bfx_{\neq n}) & \cdots & g_n(\bfx_n | \bfx_{\neq n})
    \end{pmatrix},
\end{equation*}
where $g_i(\bfx_j | \bfx_{\neq j})$ is symmetric in $\bfx_{\neq j}:=(\bfx_1,\dots,\bfx_{j-1},\bfx_{j+1},\dots,\bfx_n)$. However, when $d>1$, $g_1,g_2,\dots,g_n$ may be discontinuous even if $f$ is continuous. In contrast, in our representation, $g$ is proved as continuous when $f$ is. The continuity is crucial when modeling $g$ or $g_1,g_2,\dots,g_n$ by neural networks, as it is usually easier to train a neural network when the target function has better regularity and it is well known that fully connected neural networks are dense in the space of continuous functions on a compact domain equipped with $L^\infty$-norm. We also remark that it is possible that $g$ might not admit regularity beyond continuity, such as Lipschitz continuity or continuous differentiability, even if $f$ is Lipschitz continuous or continuously differentiable. This will be discussed in detail in Section~\ref{sec:sing_locus}.

 \subsection{Validation of Assumption~\ref{asp:eta}}
 \label{sec:construct_eta}

 In this subsection we construct $\bmeta = (\eta_1,\eta_2,\dots,\eta_m)$ satisfying Assumption~\ref{asp:eta}. The construction of $\bmeta$ is based on a set of symmetric functions that can distinguish different orbits of $S_n$.

 \begin{assumption}\label{asp:psi}
     Given $d\geq 1$ and $n\geq 1$. Assume that $\bmpsi = (\psi_1,\psi_2,\dots,\psi_q):(\bR^d)^n\to\bR^q$ satisfies the followings:
     \begin{itemize}
         \item[(i)] $\psi_k:(\bR^d)^n\to\bR$ is symmetric and continuous for every $k\in\{1,2,\dots,q\}$.
         \item[(ii)] $\bmpsi(\bfx_1,\bfx_2,\dots,\bfx_n) = \bmpsi(\bfx_1',\bfx_2',\dots,\bfx_n')$ if and only if there exists some $\sigma\in S_n$ such that $\bfx_{\sigma(i)} = \bfx_i',\ \forall~i\in\{1,2,\dots,n\}$.
         \item[(iii)] $\bmpsi(\bfx_1,\bfx_2,\dots,\bfx_n)\neq 0$ for all $\bfx_1,\bfx_2,\dots,\bfx_n\in\bR^d$.
     \end{itemize}
 \end{assumption}

\begin{remark}
    In Assumption~\ref{asp:psi}, Condition (iii) is automatically satisfied as long as at least one $\psi_k$ is a nonzero constant function. Consequently, the central requirement is Condition~(ii), namely that $\bmpsi$ separates the orbits of the $S_n$-action. Several constructions of such orbit-separating maps $\bmpsi$ have appeared in the literature. The most direct approach is to employ generators of symmetric polynomials, for instance multisymmetric power sums; see \cite{chen2022representation}. This yields $q = \binom{n+d}{d}$,
    which scales polynomially in $n$ when the particle dimension $d$ is fixed, a regime consistent with many applications, including many-body physics. Beyond this classical construction, \cite{wang2023polynomial} establishes the existence of functions $\psi_1,\psi_2,\ldots,\psi_q$ satisfying Assumption~\ref{asp:psi} with $q=\mathrm{poly}(n,d)$. Moreover, \cite{dym2024low} proves that linear scaling in both $n$ and $d$ suffices, say $q=2nd+2$.
\end{remark}

The next proposition verifies Assumption~\ref{asp:eta} with an explicit number of anti-symmetric basis functions.

\begin{proposition}\label{prop:construct_eta}
    Suppose that Assumption~\ref{asp:psi} holds, then there exist anti-symmetric and continuous functions $\eta_1,\eta_2,\dots,\eta_m$ satisfying Assumption~\ref{asp:eta} with $m = \left(\frac{n(n-1)}{2} \cdot(d-1) + 1\right) \cdot q$ and $\Omega = \bR^d$.
\end{proposition}

We need the following two lemmas for proving Proposition~\ref{prop:construct_eta}. The idea of considering the one-dimensional projections of $\bfx_1,\bfx_2,\dots,\bfx_n$ is from \cite{wang2023polynomial}.

\begin{lemma}\label{lem:construct_w}
    Let $\bfw_1,\bfw_2,\dots,\bfw_p\in\bR^d$ with $p= \frac{n(n-1)}{2} \cdot(d-1) + 1$ satisfy that any $d$ vectors among them form a basis of $\bR^d$. Then for any $(\bfx_1,\bfx_2,\dots,\bfx_n)\in (\bR^d)^n$ with $\bfx_i\neq \bfx_j,\ \forall~i\neq j$, there exists $k\in\{1,2,\dots,p\}$ such that $\bfw_k^\top \bfx_i \neq \bfw_k^\top \bfx_j,\ \forall~i\neq j$.
\end{lemma}

\begin{proof}
    The proof is modified from the proof of Lemma C.4 in \cite{wang2023polynomial}.
    Given any $i,j\in\{1,2,\dots,n\}$ with $i\neq j$, or equivalently $\bfx_i\neq \bfx_j$, and any $d$ vectors $\bfw_{k_1},\bfw_{k_2},\dots,\bfw_{k_d}$ that form a basis, there exists some $s\in\{1,2,\dots,d\}$ such that $\bfw_{k_s}^\top \bfx_i \neq \bfw_{k_s}^\top \bfx_j$; otherwise one has $\bfx_i = \bfx_j$ and this is a contradiction. Therefore, it holds that
    \begin{equation*}
        \#\left\{k\in\{1,2,\dots,p\} : \bfw_k^\top \bfx_i = \bfw_k^\top \bfx_j\right\} \leq d-1,
    \end{equation*}
    for all pair $(i,j)$ with $i\neq j$. Since $p > \frac{n(n-1)}{2} \cdot(d-1)$, there exists at least one $k\in \{1,2,\dots,p\}$ such that $\bfw_k^\top \bfx_i \neq \bfw_k^\top \bfx_j,\ \forall~i\neq j$.
\end{proof}

\begin{lemma}\label{lem:construct_phi}
    Let $p= \frac{n(n-1)}{2} \cdot(d-1) + 1$. There exist anti-symmetric and continuous functions $\phi_1,\phi_2,\dots,\phi_p$ such that for any $\bfx = (\bfx_1,\bfx_2,\dots,\bfx_n)\in(\bR^d)^n$, $\phi_1(\bfx) = \phi_2(\bfx) = \cdots = \phi_p(\bfx) = 0$ if and only if $\bfx_i=\bfx_j$ for some $i,j\in\{1,2,\dots,n\}$ with $i\neq j$.
\end{lemma}

\begin{proof}
    Let $\bfw_1,\bfw_2,\dots,\bfw_p\in\bR^d$ be vectors as in Lemma~\ref{lem:construct_w}. Define for any $k\in\{1,2,\dots,p\}$ that
    \begin{equation*}
        \phi_k(\bfx) = \det\begin{pmatrix}
            1 & \bfw_k^\top \bfx_1 & \cdots & (\bfw_k^\top \bfx_1)^{n-1} \\
            1 & \bfw_k^\top \bfx_2 & \cdots & (\bfw_k^\top \bfx_2)^{n-1} \\
            \vdots & \vdots & \ddots & \vdots \\
            1 & \bfw_k^\top \bfx_n & \cdots & (\bfw_k^\top \bfx_n)^{n-1}
        \end{pmatrix}
        = \prod_{1\leq i<j\leq n} (\bfw_k^\top \bfx_j - \bfw_k^\top \bfx_i),
    \end{equation*}
    which is the Vandermonde matrix and is an anti-symmetric and continuous function. Consider any $(\bfx_1,\bfx_2,\dots,\bfx_n)\in (\bR^d)^n$ with $\bfx_i\neq \bfx_j,\ \forall~i\neq j$. By Lemma~\ref{lem:construct_w}, there exists $k\in\{1,2,\dots,p\}$ such that $\bfw_k^\top \bfx_i \neq \bfw_k^\top \bfx_j,\ \forall~i\neq j$, which implies that $\phi_k(\bfx) \neq 0$. Conversely, if $\bfx_i = \bfx_j$ for some $i\neq j$, then it is clear that $\phi_1(\bfx) = \phi_2(\bfx) = \cdots = \phi_p(\bfx) = 0$ by Remark~\ref{rmk:x_i=x_j}.
\end{proof}

Now we can present the proof of Proposition~\ref{prop:construct_eta}.

\begin{proof}[Proof of Proposition~\ref{prop:construct_eta}]
    Let $p= \frac{n(n-1)}{2} \cdot(d-1) + 1$ and let $\phi_1,\phi_2,\dots,\phi_p$ be the anti-symmetric and continuous functions in Lemma~\ref{lem:construct_phi}. Define $\bmeta = (\eta_1,\eta_2,\cdots,\eta_m)$ with $m = pq$ as follows:
    \begin{equation*}
        \bmeta = (\phi_1 \bmpsi, \phi_2\bmpsi,\dots, \phi_p\bmpsi ),
    \end{equation*}
    where $\bmpsi = (\psi_1,\psi_2,\dots,\psi_q)$ collects symmetric and continuous functions satisfying Assumption~\ref{asp:psi}. We then verify the three conditions in Assumption~\ref{asp:eta}.
    \begin{itemize}
        \item Condition (i) holds since each $\phi_k$ is anti-symmetric and each $\psi_l$ is symmetric.
        \item According to Assumption~\ref{asp:psi} (iii), $\bmeta(\bfx) = \bmeta(\bfx_1,\bfx_2,\dots,\bfx_n)=0$ if and only if $\phi_1(\bfx) = \phi_2(\bfx) = \cdots = \phi_p(\bfx) = 0$. Then Condition (ii) follows directly from Lemma~\ref{lem:construct_phi}.
        \item Suppose that $\bmeta(\bfx_1,\bfx_2,\dots,\bfx_n) = \bmeta(\bfx_1',\bfx_2',\dots,\bfx_n') \neq 0$. By Condition (ii), we have $\bfx_i\neq \bfx_j$ and $\bfx_i'\neq\bfx_j'$ for any pair $(i,j)$ with $i\neq j$. It then follows from Lemma~\ref{lem:construct_phi} that $\phi_k(\bfx) = \phi_k(\bfx')\neq 0$, where $\bfx = (\bfx_1,\bfx_2,\dots,\bfx_n)$ and $\bfx' = (\bfx_1',\bfx_2',\dots,\bfx_n')$, for some $k\in\{1,2,\dots,p\}$, which combined with $\phi_k(\bfx)\bmpsi(\bfx) = \phi_k(\bfx')\bmpsi(\bfx')$ yields that $\bmpsi(\bfx) = \bmpsi(\bfx')$. One can thus apply Assumption~\ref{asp:psi} (ii) to conclude that there exists some $\sigma\in S_n$ such that $\bfx_{\sigma(i)} = \bfx_i',\ \forall~i\in\{1,2,\dots,n\}$, which proves Condition (iii) in Assumption~\ref{asp:eta}.
    \end{itemize}
    The proof is hence completed.
\end{proof}

\begin{remark}
    After the initial post of this paper, Lemma~\ref{lem:construct_w} and Lemma~\ref{lem:construct_phi} are improved with $p=dn+1$ in~\cite{ye2024widetilde}. Furthermore, it is established in~\cite{dym2024equivariant} that the minimal $p$ such that Lemma~\ref{lem:construct_w} and Lemma~\ref{lem:construct_phi} hold is $p=n(d-1)$. These works can thus improve the choice of $p$ in the proof of Proposition~\ref{prop:construct_eta}.
\end{remark}

\subsection{Singular locus of the anti-symmetric basis functions} 
\label{sec:sing_locus}
We discuss the singular locus of $\bmeta:(\bR^d)^n\to\bR^m$ in this subsection. The following proposition proves that the set of singular points of $\bmeta$ contains
\begin{equation}\label{eq:sing}
    \mathbf{Sing}_{d,n}:=\left\{\bfx =(\bfx_1,\bfx_2,\dots,\bfx_n)\in (\bR^{d})^n\ |\ \exists~i_1\neq i_2,\ \bfx_{i_1}=\bfx_{i_2}\right\}.
\end{equation} 

\begin{proposition}\label{prop:singular1}
    Given $d,n\geq 1$ and let $\bmeta:(\bR^d)^n\to\bR^m$ be anti-symmetric and continuously differentiable. For any $\bfx\in \mathbf{Sing}_{d,n}$, the Jacobian matrix $J\bmeta(\bfx)\in\bR^{m\times nd}$ is column-rank-deficient.
\end{proposition}

\begin{proof}
    Assume that $\bfx_{i_1} = \bfx_{i_2}$ for some $1\leq i_1<i_2\leq n$, and we have by Remark~\ref{rmk:x_i=x_j} that $\bmeta(\bfx) = 0$. For any $j\in\{1,2,\dots,d\}$, it can be computed that
    \begin{align*}
        \frac{\partial\bmeta(\bfx)}{\partial x_{i_1,j}} & = \lim_{t\rightarrow 0} \frac{\bmeta(\bfx_1,\dots,\bfx_{i_1-1}, \bfx_{i_1} + t \bfe_j, \bfx_{i_1+1},\dots, \bfx_{i_2-1},\bfx_{i_2},\bfx_{i_2+1},\dots,\bfx_n) - 0}{t} \\
    & = - \lim_{t\rightarrow 0} \frac{\bmeta(\bfx_1,\dots,\bfx_{i_1-1}, \bfx_{i_1}, \bfx_{i_1+1},\dots, \bfx_{i_2-1},\bfx_{i_2} + t \bfe_j,\bfx_{i_2+1},\dots,\bfx_n) - 0}{t} \\
    & = - \frac{\partial\bmeta(\bfx)}{\partial x_{i_2,j}},
    \end{align*}
    where $x_{i,j}$ is the $j$-th entry of $\bfx_i$ and $\bfe_j\in\bR^d$ has the $j$-th entry being $1$ and all other entries being $0$. This gives two linearly dependent columns of $J\bmeta(\bfx)$ and we can thus conclude that $J\bmeta(\bfx)$ is column-rank-deficient.
\end{proof}

Though it is in general unclear whether $\mathbf{Sing}_{d,n}$ is exactly the singular locus of all $\bmeta:(\bR^d)^n\to\bR^m$ satisfying Assumption~\ref{asp:eta}, this is true for the specific $\bmeta$ constructed in the proof of Proposition~\ref{prop:construct_eta}, with $\bmpsi = (\psi_1,\psi_2,\dots,\psi_q):(\bR^d)^n\to\bR^q$ generating the $\bR$-algebra of all symmetric polynomials. In particular, we consider the following polynomial map
\begin{equation}\label{eq:construct_eta}
    \bmeta = (\bmphi, \psi_1 \bmphi, \psi_2\bmphi,\dots, \psi_q\bmphi ) : (\bR^d)^n\to\bR^m = \bR^{p(q+1)},
\end{equation}
where $\bmphi=(\phi_1,\phi_2,\dots,\phi_p)$ collects anti-symmetric polynomials as in Lemma~\ref{lem:construct_phi} and $\bmpsi = (\psi_1,\psi_2,\dots,\psi_q)$ denotes generators of symmetric polynomials.

\begin{proposition}\label{prop:singular2}
    For any $d,n\geq 1$, consider $\bmeta:(\bR^d)^n\to\bR^m$ as in \eqref{eq:construct_eta}. $J\bmeta(\bfx)$ is of full-column-rank for any $\bfx\notin \mathbf{Sing}_{d,n}$.
\end{proposition}

\begin{proof}
    By performing elementary row operations on $J\bmeta(\bfx)$, we have that
    \begin{align*}
        J\bmeta(\bfx) = \begin{pmatrix}
            J\bmphi(\bfx) \\ \psi_1(\bfx) J\bmphi(\bfx) + \bmphi(\bfx) \nabla \psi_1(\bfx)^\top \\ \vdots \\ \psi_q(\bfx) J\bmphi(\bfx) + \bmphi(\bfx) \nabla \psi_q(\bfx)^\top
        \end{pmatrix}
        \to \begin{pmatrix}
            J\bmphi(\bfx) \\ \bmphi(\bfx) \nabla \psi_1(\bfx)^\top \\ \vdots \\ \bmphi(\bfx) \nabla \psi_q(\bfx)^\top
        \end{pmatrix}
        \to \begin{pmatrix}
            J\bmphi(\bfx) \\ \phi_1(\bfx) J\bmpsi(\bfx) \\ \vdots \\ \phi_p(\bfx) J\bmpsi(\bfx)
        \end{pmatrix}.
    \end{align*}
    Since $\bfx\notin \mathbf{Sing}_{d,n}$, Lemma~\ref{lem:construct_phi} guarantees that $\phi_{i_0}(\bfx)\neq 0$ for some $i_0\in\{1,2,\dots,p\}$, and \cite{chen2022representation}*{Theorem 3.1} proves that $J\bmpsi(\bfx)$ is of full-column-rank. Therefore, $J\bmeta(\bfx)$ is also of full-column-rank, which completes the proof.
\end{proof}

Combining Proposition~\ref{prop:singular1} and Proposition~\ref{prop:singular2}, one can conclude that for the anti-symmetric base function $\bmeta:(\bR^d)^n\to\bR^m$ in \eqref{eq:construct_eta}, the singular locus is exactly given by $\mathbf{Sing}_{d,n}$ in \eqref{eq:sing}, which is consistent with the symmetric setting \cite{chen2022representation}.

We then show with examples that even if $g$ inherits the continuity of $f$, as in Theorem~\ref{thm:main}, it may not inherit higher regularity like Lipschitz continuity and continuous differentiability, especially at the singular locus $\mathbf{Sing}_{d,n}$. In the examples below, we always set $d=1$, $n=2$, $\Omega = [-1,1]$, and
\begin{align*}
    & \bmphi(\bfx) = \bmphi(x_1,x_2) = x_1 - x_2, \\
    & \bmpsi(\bfx) = \bmpsi(x_1,x_2) = \left(x_1+x_2, x_1^2 + x_2^2\right), \\
    & \bmeta(\bfx) = \bmeta(x_1,x_2) = \left((x_1-x_2), (x_1-x_2)(x_1+x_2), (x_1-x_2)(x_1^2+x_2^2)\right).
\end{align*}
It is straightforward to verify Assumption~\ref{asp:eta} for the above $\bmeta$.

\medskip
\paragraph{\textbf{Counterexample for Lipschitz continuity}} 
    Consider
    \begin{equation*}
        f(\bfx) = f(x_1,x_2) = |x_1| - |x_2|,
    \end{equation*}
    which is anti-symmetric and Lipschitz continuous. Let $\bfx = (2\epsilon, 0)$ and $\bfx' = (\epsilon,-\epsilon)$, with $\epsilon>0$ being sufficiently small. Then it can be computed that
    \begin{equation*}
        g(\bmeta(\bfx)) = f(\bfx) = 2\epsilon,\quad g(\bmeta(\bfx')) = f(\bfx') = 0,
    \end{equation*}
    and
    \begin{equation}\label{eq:ex-eta}
        \bmeta(\bfx) = \left(2\epsilon,4\epsilon^2, 8\epsilon^3\right),\quad \bmeta(\bfx') = \left(2\epsilon, 0, 4\epsilon^3\right).
    \end{equation}
    Thus, we have that
    \begin{equation*}
        \lim_{\epsilon\to 0+} \frac{|g(\bmeta(\bfx)) - g(\bmeta(\bfx'))|}{\|\bmeta(\bfx) - \bmeta(\bfx')\|} = \lim_{\epsilon\to 0+} \frac{2\epsilon}{\|(0,4\epsilon^2,4\epsilon^3)\|} = \lim_{\epsilon\to 0+} \frac{1}{\|(0,2\epsilon,2\epsilon^2)\|} = +\infty,
    \end{equation*}
    which implies that $g$ is not Lipschitz continuous.

\medskip
\paragraph{\textbf{Counterexample for continuous differentiability}}
    Consider
    \begin{equation*}
        f(\bfx) = f(x_1,x_2) = x_1^{4/3} - x_2^{4/3},
    \end{equation*}
    which is anti-symmetric and continuously differentiable. Let $\bfx$ and $\bfx'$ be the same as before, which gives \eqref{eq:ex-eta} and
    \begin{equation*}
        g(\bmeta(\bfx)) = f(\bfx) = 2^{4/3} \epsilon^{4/3},\quad g(\bmeta(\bfx')) = f(\bfx')=0.
    \end{equation*}
    Assume that $g:\eta(\Omega^2)\to\bR$ can be extended to a continuously differentiable function on $\bR^3$, which is still denoted as $g$ for simplicity. Then there exists $\bmxi\in\bR^3$ located on the line segment connection $\bmeta(\bfx)$ and $\bmeta(\bfx')$, such that
    \begin{equation*}
        2^{4/3} \epsilon^{4/3} = g(\bmeta(\bfx)) - g(\bmeta(\bfx')) = \left\langle \nabla g(\bmxi), \bmeta(\bfx) - \bmeta(\bfx') \right\rangle = \left\langle \nabla g(\bmxi), (0,4\epsilon^2,4\epsilon^3) \right\rangle,
    \end{equation*}
    which leads to that
    \begin{equation*}
        \left\langle \nabla g(0,0,0), (0,4,0) \right\rangle = \lim_{\epsilon\to 0+} \left\langle \nabla g(\bmxi), (0,4,4\epsilon) \right\rangle = \lim_{\epsilon\to 0+} \frac{2^{4/3}}{\epsilon^{2/3}} = +\infty,
    \end{equation*}
    which is a contradiction. Therefore, $g$ cannot be extended to a continuously differentiable function on $\bR^3$.

\bibliographystyle{amsxport}
\bibliography{references}

\end{document}